\documentclass[12pt]{article}

\usepackage{amsthm,amsfonts,amsmath,amscd,amssymb,times,epsfig,verbatim}
\usepackage[all]{xy}
\usepackage{enumerate,array}

\newtheorem{thm}{Theorem}
\newtheorem{prop}{Proposition}
\newtheorem{lem}{Lemma}
\newtheorem{cor}{Corollary}

\theoremstyle{remark}
\newtheorem{rem}{Remark}
\theoremstyle{definition}
\theoremstyle{axiom}

\newtheorem{defn}{Definition}

\newcommand{\C}{\mathbb{C}}

\newcommand{\R}{\mathbb{ R}}

\newcommand{\Z}{\mathbb{ Z}}

\title{ The orbit spaces $G_{n,2}/T^n$  and the Chow quotients  $G_{n,2}\!/\!/(\C ^{\ast})^{n}$  of the Grassmann manifolds $G_{n,2}$}

\author{Victor M.~Buchstaber and Svjetlana Terzi\'c}

\begin{document}
\maketitle
\begin{abstract}
The  focus of our paper is on  the complex Grassmann manifolds $G_{n,2}$  which appear as one of the fundamental objects  in developing the interaction between algebraic geometry and algebraic topology.
In his  well-known paper  Kapranov   has  proved that the Deligne-Mumford compactification   $\overline{\mathcal{M}}(0,n)$ of $n$-pointed curves  of genus zero  can be realized as the Chow quotient $G_{n,2}\!/\!/(\C ^{\ast})^{n}$.  
  In our recent papers,   the constructive description of the orbit space $G_{n,2}/T^n$ has been obtained. In getting this result   our  notions of the CW-complex of the admissible polytopes and the universal space of parameters $\mathcal{F}_{n}$ for  $T^n$-action on $G_{n,2}$ were of essential use.  Using technique of the wonderful compactification, in this paper   it is given an explicit construction of the space $\mathcal{F}_{n}$. Together with   Keel's  description  of $\overline{\mathcal{M}}(0,n)$, this construction  enabled us to obtain   an explicit   diffeomorphism between  $\mathcal{F}_{n}$ and  $\overline{\mathcal{M}}(0,n)$.  Thus,  we showed  that the space $G_{n,2}\!/\!/(\C ^{\ast})^{n}$ can be realized as our universal space of parameters $\mathcal{F}_{n}$.   
In this way,  we give   description  of the structure in  $G_{n,2}\!/\!/(\C ^{\ast})^{n}$, that is $\overline{\mathcal{M}}(0,n)$    in terms of the CW-complex of the admissible polytopes for $G_{n,2}$ and their spaces of parameters.


\end{abstract} 

\newpage

\tableofcontents

\section{Introduction}
The questions about  the action  of the maximal  algebraic torus as well as the induced action of the compact torus on the complex Grassmann manifolds  naturally arise in many areas of mathematics, ~\cite{GS}, ~\cite{GKZ},~\cite{GeM},~\cite{GM},~\cite{Kap}. It is well known that the Grassmann manifolds are  a fundamental object for the problems which develop an interaction between algebraic geometry and algebraic topology. The subject of toric geometry are algebraic manifolds which consist of the closure  of one orbit of algebraic torus.  The equivariant structure of these manifolds can be effectively described in terms of combinatorial structure of a moment polytope. The issue which naturally arises in this context is to study   algebraic manifolds  with algebraic torus action  whose family of algebraic torus orbits   gives, 
in a way,  good stratification of a manifold. The first such nontrivial examples are the Grassmann manifolds $G_{n,2}$ which can be as well  interpreted as the space of all projective lines in $\C P^{n-1}$. In addition, from  the results of Gel'fand-Serganova~\cite{GS}  it follows  that the structure of  the algebraic torus orbit  stratification, that is the gluing of  the strata, in the case of $G_{n,2}$  can be explicitly described. This is  no more true for the Grassmannians $G_{n,k}$,  $n\geq 7$,  $k\geq 3$,  as it is demonstrated on the example   $G_{7,3}$ in~\cite{GS}, see also~\cite{BT-2nk}.   Further on, we generalized  this property of $G_{n,2}$  and introduced  in~\cite{BT-2nk}  the notion of the CW complex of the strata. 

In this paper we  establish the connection between the well-known constructions  from algebraic geometry based on the notion of  wonderful compactification  and the results on equivariant algebraic topology of the Grassmann manifolds.

The motivation to tackle  this issue has come  to us from the problem of   describing  the universal space of parameters $\mathcal{F}_{n}$  for the canonical $T^n$-action on  a Grassmannian $G_{n,2}$.  The universal spaces of parameters have been introduced in our theory of $(2m, k)$-manifolds,  see~\cite{BT-2nk}. This theory focuses on  smooth $2m$-dimensional manifolds with a smooth action of $k$-dimensional compact torus,  which satisfy   given set of axioms. The important class of such manifolds  is consisted  of the manifolds for which an action of a compact torus is induced by the  action of an algebraic torus. Among these manifolds a beautiful role  have  the manifolds $G_{n,2}$,   which are $(2m,k)$-manifold for $m=2(n-2)$ and $k=n-1$.  The universal space of parameters $\mathcal{F}$  for a $(2m,k)$-manifold $M^{2m}$ with an effective action of the compact torus $T^k$, $k\leq m$, is a compactification of the space of parameters $F$ of the main stratum.  Its general definition is given in~\cite{BT-2nk}.
The  universal space of parameters  for the Grassmannian $G_{5,2}$  in regard with the canonical action of $T^5$  is defined and explicitly described in~\cite{BT-2}. More precisely,  it is proved that in this case the  universal space of parameters is  the blow up of the cubic surface $c_1c_2^{'}c_3=c_1^{'}c_2c_{3}^{'}$ in $(\C P^{1})^{3}$ at one point, being the same as the blow up of $\C P^{2}$ at four points in general position. This space is in algebraic geometry known as the del Pezzo surface of degree $5$, see also ~\cite{HS}. Further on,     in~\cite{klem}  it is proved  that the  Grothendieck-Knudsen compactification  of the moduli space $\mathcal{M}(0, n)$  of genus zero curves with $n$ marked distinct points
 provides   universal space  of parameters for $G_{n,2}$ in regard to the canonical $T^n$ -action.    
From the one side,  the method  in~\cite{klem}  appeals  on the Gel'fand MacPherson correspondence ~\cite{GeM}  between the space $F_n$ and configurations of all pairwise distinct points in $(\C P^1)^{n}$ up to automorphism of $\C P^1$ and, from the other side,  on the description of the Grothendieck-Knudsen compactification $\overline{\mathcal{M}}(0,n)$  of the moduli space $\mathcal{M}(0, n)$  of genus zero curves with $n$ marked distinct points via cross ratios as  described by McDuff-Salamon  in~\cite{MS}. There is also the notion of the Chow quotient $G_{n,2}\!/\!/ (\C ^{\ast})^{n}$ defined by  Kapranov in~\cite{Kap} for which he  proved to coincide  with  the space $\overline{\mathcal{M}}(0,n)$. It is well known fact that $\overline{\mathcal{M}}(0,5)$ is the del Pezzo surface of degree $5$, while  algebro-geometric characterization of the space $\overline{\mathcal{M}}(0,n)$  for $n>5$ is an open well known problem.
In the paper~\cite{BV} it has been shown that   the problem of   algebro-topological  characterization of these spaces is related to  well known problems of complex cobordisms theory.

 In this paper we provide new  description of the universal space of parameters $\mathcal{F}_{n}$ for $G_{n,2}$  which   comes purely    from equivariant topology of the  Grassmann manifolds $G_{n,2}$. We obtain a smooth, compact manifold for which we prove to be diffeomorphic to the moduli space $\overline{\mathcal{M}}(0,n)$ of genus zero stable curves with $n$ marked distinct points, that is to the Chow quotient $G_{n,2}\!/\!/ (\C ^{\ast})^{n}$.

In this context, the main task  to be considered can be formulated as follows: given an  algebraic variety $X$  in  $(\C P^1)^{N}$ which is  an open subset  of its closure $\bar{X}$ in $(\C P^{1})^{N}$,   and a family of its  automorphism $\mathcal{A}(X)$, to  find a compactification $\mathcal{X}$ of $X$  for which there is the projection $p: \mathcal{X}\to \bar{X}$ which  restriction on $X$ is an identity, and  any automorphism $f\in \mathcal{A}(X)$ extends to the automorphism of $\mathcal{X}$.  

Here we take  $X$ to be the coordinate record  of the space of parameters $F_n$ of the main stratum in a   standard chart for  $G_{n,2}$ which is defined by the Pl\"ucker coordinates, while for  $\mathcal{A}(F_n)$ we take all automorphism of $F_n$ induced by the coordinate changes between  a fixed chart and  all other charts.  In order to find $\mathcal{F}_n$ we resolve the   singularities of the closure $\bar{F}_n \subset (\C P^{1})^{N}$, which arise in the context of   extension of the automorphisms from $\mathcal{A}(F_n)$, using the constructions of algebraic geometry known as    wonderful compactification. We prove that in  such  a way  the smooth manifold $\mathcal{F}_n$  is obtained, which is diffeomorphic to the Chow quotient    $G_{n,2}\!/\!/ (\C ^{\ast})^{n}$ that is,  to the Grothendieck-Knudsen compactification of the space $\mathcal{M}(0,n)$. 

The wonderful compactification of a complex manifold $M$ is a compact manifold $\tilde{X}$, such that $D=\tilde{X}\setminus M$ is a divisor with  normal crossings in  $\tilde{X}$ whose irreducible components are smooth,  and any number of connected components of $D$ intersect transversally. Such strong conditions on compactification turn out to be of essential importance for many algebraic and geometric problems,   such as  the problems of enumerative algebraic geometry and Schubert enumerative problems, description of the rational homotopy type of $M$, its mixed Hodge structure, the  Chow ring,  etc.   The notion of wonderful compactification  firstly appeared  in the paper~\cite{CP} of De Concini-Procesi  in the context  of  an equivariant compactification of the symmetric spaces $G/H$, see also~\cite{LV} and~\cite{T} for a comprehensive overview of the subject.  This idea has been further developed and applied  in many directions, such as  Fulton-MacPherson compactification  in~\cite{FM}, De Concini-Procesi wonderful models~\cite{CP},~\cite{CP2}, the wonderful compactification of Li~\cite{LL} and  more recently  the  projective wonderful models of toric arrangements  by De Concini-Gaiffi and others~\cite{CG1},~\cite{CG2}, ~\cite{CGP}.

Furthermore, this paper  finds    the advantage of wonderful compactification for the description of the equivariant topology of the Grassmannians $G_{n,2}$ for the canonical $T^n$-action. More explicitly, it turns out that  the wonderful compactification of arrangements of subvarieties from~\cite{LL} can be successfully applied for the compactification of the space of parameters $F_n$ of the main stratum $W_n\subset G_{n,2}$. In this way,  we     obtain the  smooth manifold $\mathcal{F}_{n}$  which enables us to construct the   
model $U_n = \mathcal{F}_{n}\times \Delta_{n,2}$ for the orbit space $G_{n,2}/T^n$ meaning that there exists a continuous surjection $p_n : U_n \to G_{n,2}/T^n$, see~\cite{BT-n2}.

This Chow quotient is the compactification of the orbit space $W_n/(\C ^{\ast})^{n}$ for the main stratum $W_n\subset G_{n,2}$  given in terms of the Chow variety  for $G_{n,2}$, see~\cite{Kap},~\cite{GKZ}.  The build up components for this compactification are described in~\cite{Kap} in terms of the maximal algebraic torus orbits in $G_{n,2}$. In this paper, we push this further and   show that   the build up components of $G_{n,2}\!/\!/ (\C ^{\ast})^{n}$ can be described in terms of the CW complex of admissible polytopes and their spaces of parameters.  In this description, following the ideas of the Chow quotient from~\cite{Kap}, the key role have the cort\'eges of $(n-1)$-dimensional  admissible polytopes which give polyhedral decompositions for $\Delta_{n,2}$.  We relate this to  our description  of the space $\mathcal{F}_{n}$ and provide an explicit expression  for  the build up  components of  $G_{5,2}\!/\!/ (\C ^{\ast})^{5}$.

\section{Space of parameters $F_n$ of the main stratum  for   $G_{n,2}$}

\subsection{ The embedding of   $F_n$ in $(\C P^{1})^{N}$}
First we  recall  from~\cite{BT-2} and~\cite{BT-2nk} the notions of the main stratum $W_n$ in $G_{n,2}$ as well as  its space of parameters  $F_n$  as a background introduction into the objects we are going to consider.

The main stratum $W_n\subset G_{n,2}$ is characterized by the condition that its points have all  non-zero  Pl\"ucker coordinates, which implies that  $W_n$  belongs to any standard chart $M_{ij} = \{ L\in G_{n,2} |  P^{ij}(L)\neq 0\}$, $1\leq i<j\leq n$ of a  Grassmann $G_{n,2}$.  The main stratum $W_n$ is invariant under the canonical action of the algebraic torus $(\C ^{\ast})^{n}$ and the orbit space 
$F_{n} = W_n/(\C ^{\ast})^{n}$ is said to be the space of parameters for $W_n$.  In a fixed chart, the main stratum  is given    by the following  system of equations:
\begin{equation}\label{orbit}
c_{ij}^{'}z_iw_{j} = c_{ij}z_jw_{i}, \;\; 3\leq i<j\leq n,
\end{equation} 
where the parameters $(c_{ij}:c_{ij}^{'})\in \C P^1$  and 
$c_{ij}, c_{ij}^{'}\neq 0$ and $c_{ij}\neq c_{ij}^{'}$ for all $3\leq i<j\leq n$.  

The number of parameters $(c_{ij} : c_{ij}^{'})$  is $N={n-2 \choose 2}$ and  from~\eqref{orbit} it follows that these parameters satisfy the following equations:
\begin{equation}\label{relat}
 c_{ij}^{'}c_{ik}c_{jk}^{'} = c_{ij}c_{ik}^{'}c_{jk}, \;\; 3\leq i<j<k\leq n.
\end{equation}
The number of these equations is ${n-2\choose 3}$.

We see   from~\eqref{relat} that  the parameters $(c_{ij}: c_{ij}^{'})$ satisfy the following relations:
\begin{equation}\label{relatmain1}
(c_{ij}:c_{ij}^{'}) = (c_{3i}^{'}c_{3j}: c_{3i}c_{3j}^{'}), \;\; 4\leq i<j\leq n.
\end{equation}
\begin{equation}\label{relatmain2}
 (c_{3i}:c_{3i}^{'})\neq (c_{3j}:c_{3j}^{'}), \;\;  4\leq i<j\leq n.
\end{equation}
The number of these relations  is $M={n-3\choose 2}$.
 Therefore, we obtain  that the space $F_n$ is homeomorphic  to the space
\begin{equation}\label{spmain}
F_n\cong  (\C P^{1}_{A})^{n-3}\setminus \Delta, 
\end{equation}
where $A=\{(0:1), (1:0), (1:1)\}$  and $\Delta = \cup _{3\leq i<j\leq n}\Delta _{ij}$ for the diagonals $\Delta _{ij} = \{((c_{34}:c_{34}^{'}), \ldots , (c_{n-1,n}, c_{n-1,n}^{'}))\in (\C P^{1}_{A})^{n-3} | (c_{3i}:c_{3i}^{'}) =  (c_{3j}:c_{3j}^{'})\}$.

The relations ~\eqref{relat} give  an embedding of  the space  $F_n$  into $(\C P^{1})^{N}$, $N={n-2\choose 2}$, that is:
\begin{lem}
The space of parameters $F_n$ of the main  stratum $W_n$ is given by~\eqref{relat} as a subspace in $(\C P^{1})^{N}$, $N={n-2\choose 2}$.
\end{lem}

Moreover,  we see that $F_n$ is an open algebraic manifold in $\bar{F}_{n}\subset (\C P^{1})^{N}$ given by the intersection of the cubic hypersurfaces~\eqref{relat} and  the conditions that $(c_{ij}^{'}:c_{ij})\in \C P^1\setminus A$, where  $A=\{(0:1), (1:0), (1:1)\}$.  The dimension of $F_n$ is $2(n-3)$ what is exactly equal to $2(N-M)$.

\begin{prop}\label{smooth}
The compactification $\bar{F}_{n}$ of $F_{n}$ in $(\C P^{1})^{N}$ is a smooth algebraic variety given by
\[
\{((c_{ij} : c_{ij}^{'}))_{3\leq i<j\leq n}\in (\C P^{1})^{N} \; | \; c_{ij}^{'}c_{ik}c_{jk}^{'} = c_{ij}c_{ik}^{'}c_{jk}, \; 3\leq i<j<k\leq n\}.
\]
\end{prop}

\begin{proof}
We consider the gradients of the functions $f_{ijk} = c_{ij}c_{ik}^{'}c_{jk} - c_{ij}^{'}c_{ik}c_{jk}^{'}$ which define $\bar{F}_{n}$.
It can be directly verified   that   there are   $M={n-3\choose 2}$ linearly independent vectors  among these gradients at any  point of $\bar{F}_{n}$.  This  implies that $\bar{F}_{n}$ is a smooth algebraic variety of real dimension $2(N-M)$ for $N={n-2\choose 2}$ and $M={n-3\choose 2}$.
\end{proof}


\subsection{ Requirements on   the compactification  of  $F_{n}$}

The  motivation for our approach  in finding the compactification for the space of parameters $F_{n}$ of the main stratum $W_{n}$ in  $G_{n,2}$ comes from our work on description of  topology of the orbit space $G_{n,2}/T^n$. Using Pl\"ucker coordinates one can define  the stratification of a  Grassmannian $G_{n,2}$, that is  $G_{n,2}=\cup _{\sigma}W_{\sigma}$, where $\sigma \subset \{\{i,j\}\subset \{1, \ldots ,n\}, \; i\neq j\}$. The strata $W_{\sigma} = \{ L\in G_{n,2} | P^{ij}(L)\neq 0, \;\ \text{for}\; \{i,j\} \in \sigma\}$ are pairwise disjoint and $T^n$-invariant, even $(\C ^{\ast})^{n}$-invariant. This  induces the stratification of the orbit space
\[
G_{n,2}/T^n= \cup_{\sigma}W_{\sigma}/T^n.
\]
  The main stratum $W_{n}$   is a dense set in  $G_{n,2}$ and the torus $T^{n-1}=T^{n}/S^1$ acts freely on it. In~\cite{BT-2} we proved that 
\[
W_{n}/T^{n}\cong \stackrel{\circ}{\Delta}_{n,2}\times F_{n},
\]
 where $\Delta_{n,2}$ is the hypersimplex. Moreover, for any stratum $W_{\sigma}$ we proved that 
\[
W_{\sigma}/T^{n}\cong \stackrel{\circ}{P}_{\sigma}\times F_{\sigma},
\]
 where $P_{\sigma}$ is a subpolytope in $\Delta_{n,2}$  given by $P_{\sigma} = \text{convhull}\{\Lambda _{ij} = e_i+e_j\in \R ^{n}, \{i, j\}\in \sigma\}$ and $F_{\sigma} = W_{\sigma}/(\C ^{\ast})^{n}$. Note that all this implies that to any stratum $W_{\sigma}$   it can be assigned  the subspace $\tilde{F}_{\sigma}\subset \mathcal{F}_{n}$  for  any compactification $\mathcal{F}_{n}$ for $F_{n}$.  In order to describe the orbit space $G_{n,2}/T^n$ we look for such compactification $\mathcal{F}_{n}$ for $F_n$ for which $\mathcal{F}_{n}= \cup _{\sigma}\tilde{F}_{\sigma}$ and  for any stratum $W_{\sigma}$  there exists the  projection $p_{\sigma} : \tilde{F}_{\sigma}\to F_{\sigma}$. In finding such compactification    we start by noting the following:   if consider  a stratum $W_{\sigma}$  in a chart $M_{ij}$
 and assign to it  the space $\tilde{F}_{\sigma, ij}\subset \mathcal{F}_{n}$ using the fact that  $W_{\sigma}/T^{n}$ is in the boundary of $W_{n}/T^n$, then the  space $\tilde{F}_{\sigma, ij}$ must not depend on a chart $M_{ij}$ such that $W_{\sigma}\subset M_{ij}$. Thus, we start by considering $\bar{F}_{n}\subset (\C P^{1})^{N}$ and  all possible $\tilde{F}_{\sigma, ij}\subset \bar{F}_{n}$,  and make  corrections along  those $\tilde{F}_{\sigma, ij}$ which are not independent of a chart $M_{ij}$.

\subsection{ Automorphisms  of $F_n$ induced by the  changes of coordinates}
The  main stratum $W_n$ belongs to any  chart defined by the Pl\"ucker coordinates. It follows from~\eqref{orbit} and~\eqref{relat} that  the transition maps  between the charts produce  the automorphisms  of the space of parameters $F_n$ of the main stratum.

We deduce explicitly  the automorphisms  of  $F_n$  induced by the transition maps  between the charts $M_{12}$ and $M_{ij}$, $i<j$, $\{1,2\}\neq \{i,j\}$.  
Denote   the local coordinates in the chart $M_{12}$  by
\[
z_3, \ldots , z_{n},w_3, \ldots, w_{n}
\]
  and  let 
\[
z_{1}^{'}, \ldots, z_{i-1}^{'}, z_{i+1}^{'}, \ldots , z_{j-1}^{'}, z_{j+1}^{'}, \ldots, z_n^{'},\]
\[
 w_{1}^{'}, \ldots, w_{i-1}^{'}, 
w_{i+1}^{'}, \ldots , w_{j-1}^{'}, w_{j+1}^{'}, \ldots, w_n^{'}
\]
be  the local coordinates in the chart $M_{ij}$.

 Let further $(c_{pq}: c_{pq}^{'})$, $3\leq p<q\leq n$ be the coordinate record of the space $F_n$ in the chart $M_{12}$ and $(d_{kl}: d_{kl}^{'})$, $1\leq k<l\leq n$, $k,l\neq i,j$ be the coordinate record of $F_n$ in a chart $M_{ij}$.  

We differentiate the  following cases.

1) $i=1$, $3\leq j\leq n$, that is we consider the chart $M_{1j}$. In this case we have that
\begin{equation}\label{relcoord}
z_{2}^{'} = -\frac{z_j}{w_j}, \;\; z_{k}^{'} = \frac{z_{k}w_{j}-z_{j}w_{k}}{w_j}, \; k\geq 3,
\end{equation}
\[
w_{2}^{'} = \frac{1}{w_{j}}, \;\; w_{k}^{'} = \frac{w_{k}}{w_{j}}, \; k\geq 3.
\]

\begin{lem}\label{jedan}
The coordinates  $(c_{pq}: c_{pq}^{'})$ and $(d_{kl}:d_{kl}^{'})$  of the space of parameters $F_n$ in the charts $M_{12}$ and $M_{1j}$, $ 3\leq j\leq n$ are related by
\[
(d_{2l} : d_{2l}^{'}) = (c_{jl}; c_{jl}-c_{jl}^{'}), \;\; 3\leq l\leq n, \; l\neq j
\]
\[
(d_{kl}: d_{kl}^{'}) = (c_{jl}(c_{jk}-c_{jk}^{'}): c_{jk}(c_{jl}-c_{jl}^{'})) , \;\;  3\leq k<l\leq n, \; k, l\neq j.
\]
The second expression can be written as
\[
(d_{kl} : d_{kl}^{'}) = (c_{kl}c_{jl}^{'}(c_{jk}-c_{jk}^{'}): c_{kl}^{'}c_{jk}^{'}(c_{jl}-c_{jl}^{'})).
\]
\end{lem}
\begin{proof}
In the chart $M_{1j}$ the main stratum writes as
\[
d_{kl}^{'}z_{k}^{'}w_{l}^{'} = d_{kl}z_{l}w_{k}, \;\; 2\leq k<l\leq n, \; k,l\neq j.\]

Substituting the relations~\eqref{relcoord} into these equations   of the main stratum, we obtain 
\[
- d_{2l}^{'}\cdot  \frac{z_j}{w_j}\cdot \frac{w_{l}}{w_{j}} = d_{2l}\cdot \frac{z_lw_j-z_jw_l}{w_j}\cdot \frac{1}{w_j},
\]
which implies 
\[
d_{2l}^{'} = -d_{2l} (\frac{z_lw_j}{z_jw_l}-1)= d_{2l}(1-\frac{c_{jl}^{'}}{c_{jl}}),
\]
that is
\[
(d_{2l} : d_{2l}^{'}) = (c_{jl}: c_{jl} - c_{jl}^{'}).
\] 
For $k\geq 3$ in an analogous way we obtain
\[
d_{kl}^{'}w_{l}(z_kw_j-z_jw_k) = d_{kl}w_{k}(z_lw_j - z_jw_l),
\]
which implies
\[
d_{kl}^{'} =d_{kl}\cdot  \frac{w_k}{w_l}\cdot  \frac{z_jw_l (\frac{c_{jl}^{'}}{c_{jl}}-1)}{z_jw_k(\frac{c_{jk}^{'}}{c_{jk}}-1)},
\] 
that is 
\[
(d_{kl}: d_{kl}^{'}) = (c_{jl}(c_{jk}^{'}-c_{jk}): c_{jk}(c_{jl}^{'}-c_{jl})).
\]
\end{proof}

2) $i=2$, $ 3\leq j\leq n$, that is we consider the chart $M_{2j}$. We have that 
\[
z_{1}^{'} = -\frac{w_j}{z_j}, \;\; z_{k}^{'} = \frac{z_{j}w_{k}-z_{k}w_{j}}{z_j}, \; k\geq 3,
\]
\[
w_{1}^{'} = \frac{1}{z_{j}}, \;\; w_{k}^{'} = \frac{z_{k}}{z_{j}}, \; k\geq 3.
\]
Substituting this  into  equations of the main stratum written in the chart $M_{2j}$ we obtain:
\begin{lem}\label{dva}
The coordinates $(c_{pq}: c_{pq}^{'})$   and $(d_{kl}:d_{kl}^{'})$  of the space of parameters $F_n$ in the charts $M_{12}$ and $M_{2j}$, $ 3\leq j\leq n$ are related by
\[
(d_{1l} : d_{1l}^{'}) = (c_{jl}^{'}; c_{jl}^{'}-c_{jl}), \;\; 3\leq l\leq n, \; l\neq j,
\]
\[
(d_{kl}: d_{kl}^{'}) = (c_{jl}^{'}(c_{jk}-c_{jk}^{'}): c_{jk}^{'}(c_{jl}-c_{jl}^{'})) 
\]
\[
= ( c_{jl}c_{kl}^{'}(c_{jk}-c_{jk}^{'}): c_{jk}c_{kl}(c_{jl}-c_{jl}^{'})), \;\; 3\leq k< l\leq n, \; k,l\neq j
\]
\end{lem}
3) $3\leq i<j\leq n$, we are  in the chart $M_{ij}$ and 
\[
z_{1}^{'} = \frac{w_{j}}{z_{i}w_{j}-z_{j}w_{i}}, \;\; z_{2}^{'} = -\frac{z_{j}}{z_{i}w_{j}-z_{j}w_{i}},
\;\; z_{k}^{'} = \frac{z_{k}w_{j}-z_{j}w_{k}}{z_{i}w_{j}-z_{j}w_{i}}, \; k\geq 3.
\]
\[
w_{1}^{'} = \frac{w_{i}}{z_{j}w_{i}-z_{i}w_{j}}, \;\; w_{2}^{'} = -\frac{z_{i}}{z_{j}w_{i}-z_{i}w_{j}},\;\;  w_{k}^{'} = \frac{z_{k}w_{i}-z_{i}w_{k}}{z_{j}w_{i}-z_{i}w_{j}} \; k\geq 3.
\]
For the coordinates of the space of parameters $F_n$ we obtain:
\begin{lem}\label{tri}
The coordinates  $(c_{kl}: c_{kl}^{'})$ and $(d_{kl}:d_{kl}^{'})$  of the space of parameters $F_n$ in the charts $M_{12}$ and $M_{ij}$, $3\leq i<j\leq n$ are related by
\[
(d_{12}:d_{12}^{'}) = (c_{ij} : c_{ij}^{'}),
\]
\[ (d_{1l}: d_{1l}^{'}) = (c_{jl}^{'}(c_{il}-c_{il}^{'}): c_{il}^{'}(c_{jl}-c_{jl}^{'})), \;\; 3\leq l\leq n, \; l\neq i,j,
\]
\[
(d_{2l}: d_{2l}^{'}) = (c_{jl}(c_{il}-c_{il}^{'}) : c_{il}(c_{jl}-c_{jl}^{'})), \;\;   3\leq l\leq n, \; l\neq i,j
\]
\[
(d_{kl}: d_{kl}^{'}) = ((c_{jk}^{'}-c_{jk})(c_{il}-c_{il}^{'})c_{ik}^{'}c_{jl}^{'}: (c_{ik}-c_{ik}^{'})(c_{jl}-c_{jl}^{'})c_{jk}^{'}c_{il}^{'}) 
\]
\[= 
((c_{jk}-c_{jk}^{'})(c_{il}-c_{il}^{'})c_{ik}c_{jl}^{'}c_{kl}: (c_{ik}-c_{ik}^{'})(c_{jl}-c_{jl}^{'})c_{jk}^{'}c_{il}c_{kl}^{'}), \;\;  3\leq k<l\leq n, \; k,l\neq  i,j.
\]
\end{lem}

\begin{rem}\label{all}
In this way we obtain the family  $\{f_{12, ij}\}$ $i=1, j\geq 3$ or  $2\leq i<j\leq n$ of homeomorphisms of the space $F_{n}$.   Note that  any homeomorphism  $f_{kl, pq}$ of $F_{n}$  induced by the transition map between the charts $M_{kl}$ and $M_{pq}$ can be represented as $f_{kl, pq} = f^{-1}_{12, kl} \circ f_{12, pq}$. 
\end{rem}

\section{Universal space of parameters  $\mathcal{F}_{n}$ and wonderful compactification}

We want to  find the compactification $\mathcal{F}_{n}$  for $F_{n}$ such that  the  homeomorphisms $\{f_{12, ij}\}$   of $F_{n}$ induced by the transition maps  between the chart $M_{12}$ and a  chart $M_{ij}$, extend to the homeomorphisms of $\mathcal{F}_{n}$.
 According to Remark~\ref{all} it is enough to consider just the  family $\{f_{12, ij}\}$, since then
 any homeomorphism  $\{f_{kl, pq}\}$  of $F_n$ induced by the transition map  between any two charts $M_{kl}$ and $M_{pq}$ can be  canonically extended to the homeomorphism of $\mathcal{F}_{n}$.

\subsection{The compactification $\bar{F}_{n}$  of $F_n$ in $(\C P^{1})^{N}$}

We start with  $\bar{F}_{n}\subset (\C P^{1})^{N}$, $N={n-2\choose 2}$ given by
\[
\bar{F}_{n} = \{ ((c_{ij}:c_{ij}^{'})),  \; 3\leq i<j\leq n, \; c_{ik}c_{il}^{'}c_{kl}=c_{ik}^{'}c_{il}c_{kl}^{'}, \; 3\leq i<k<l\leq n\}, 
\]
 which is, by Proposition~\ref{smooth}, a smooth algebraic variety obtained as a compactification of $F_{n}$ in $(\C P^{1})^{N}$.

The compactification $\bar{F}_{n}$ is not the one we are looking for as it is obvious that the  homeomorphisms of $F_n$ defined in
previous lemmata for $n>4$  can not be  continuously  extended to $\bar{F}_n$. Note that for $n=4$  these homeomorphisms  extend to the homeomorphisms of $\bar{F}_{4} = \C P^{1}$.

In more detail, the boundary of $F_n$ in $\bar{F}_{n}$ is 
$\bar{F}_{n}\setminus F_{n}$ and it consists of the points  $(c_{ij}:c_{ij}^{'})\in \bar{F}_{n}$ such that   $c_{ij}=0$  or $c_{ij}^{'}=0$  or $c_{ij}=c_{ij}^{'}$ for some $ 3\leq i<j\leq n$. The homeomorphisms given by previous lemmata  do not continuously extend to 
the points from the boundary of $F_n$ which satisfy  $c_{jk} = c_{jk}^{'}$ and $c_{jl} = c_{jl}^{'}$ for some $3\leq j<k<l\leq n$. Moreover, even if they do continuously extend to some  subset from $\bar{F}_{n}\setminus F_{n}$ these extensions do not have to be homeomorphisms. For example, the subvariety in $\bar{F}_{n}\setminus F_{n}$ given by $c_{34}^{'} = c_{35}^{'}=0$ maps by the continuous extension of the homeomorphism  $f_{12. 13}$.  which is defined  by  Lemma~\ref{jedan} for   $j=3$,  to the subvariety given by $d_{24} =d_{24}^{'}$, $d_{25}=d_{25}^{'}$ and  $d_{45}=d_{45}^{'}$.  In particular for $n=5$ this means that the subvariety $(1:0), (1:0), (c_{45}:c_{45}^{'})$ maps to the point $((1:1), (1:1), (1:1))$, see also~\cite{BT-2}. Thus, this extension  can not  be a homeomorphism.  

To overcome these problems the idea is to blow up the smooth, compact variety $\bar{F}_{n}$ along  the singular subvarieties for $\{f_{12, ij}\}$ which consist of all  singular points  for these homeomorphisms in $\bar{F}_{n}$.   In order to  do that we use the technique from algebraic geometry  known as the wonderful compactification of an arrangement of subvarieties.

\subsection{Basic facts on wonderful compactification} 
A  wonderful compactification is a kind of compactification of a variety aimed to resolve singularities of a variety which appear in some context. The smooth   compactification $\tilde{X}$ of a  complex manifold  $M$ is wonderful in the sense that $D=\tilde{X}\setminus M$ is a divisor with normal crossings in $\tilde{X}$ whose irreducible components are smooth and any number of connected components of $D$ intersect transversally.  There are several compactifications in the literature which provide examples of wonderful compactification. We first mention  the compactification of the symmetric spaces given by De Concini and Procesi~\cite{CP1} in which $M$ is a symmetric space $G/H$ of an adjoint semisimple Lie group and $\tilde{X}$ is a smooth, compact variety with an $G$-action such that $\tilde{X}$ has an open orbit isomorphic to $G/H$ with finitely many $G$-orbits,  the orbit closures are all smooth  and any number of orbit closures intersect transversally. The other example is  the compactification of configuration spaces given by Fulton and MacPherson~\cite{FM} in which $M$ is an open subset of the Cartesian product $X^n$ of a given nonsingular variety $X$, that is  $M$ is   defined as the complement of all diagonals, while $\tilde{X}$ is defined by a sequence of blowups of $X^n$  along the nonsingular subvarieties corresponding to all diagonals. An example of wonderful compactification is  as well the compactification of arrangements of complements of linear subspaces given also by De Concini and Procesi~\cite{CP} in which $M$ is a finite-dimensional vector space and $\tilde{X}$ is obtained by replacing  any given family of its subspaces by a divisor with normal crossings. In  the recent  paper of   Li~\cite{LL}   the wonderful  compactification of arrangement of subvarieties is described and in this case  $M$ is a nonsingular variety and $\tilde{X}$ is obtained  by replacing   any  given arrangement of subvarieties by a divisor with normal crossings.


It has been proved that any  of   these  compactifications    can be  constructed by  a sequences of blow ups along appropriate subvarieties and  their transforms.     

We follow here the paper of Li~\cite{LL} to recall  the basic facts on wonderful compactification for the setting we are going to use.

\begin{defn}

Let $Y$ be a nonsingular variety over an algebraically closed field of arbitrary characteristic. Let $\mathcal{G}$ be a nonempty building set  and let $Y^{\circ} = Y\setminus \bigcup _{G\in \mathcal{G}}G$. The closure of the image of the naturally closed embedding
\[
Y^{0}\hookrightarrow \prod\limits_{G\in \mathcal{G}}Bl_{G}Y
\]
is called the wonderful compactification of $\mathcal{G}$ and it is denoted by $Y_{\mathcal{G}}$.
\end{defn}

We formulate two crucial theorems from~\cite{LL},  the first one  states  that the wonderful compactification $Y_{\mathcal{G}}$ is a nonsingular variety and the second one describes $Y_{\mathcal{G}}$ as the  series of blow-ups determined by the subvarieties from $\mathcal{G}$.  

\begin{thm}\label{t1}
Let $Y$ be a nonsingular variety and let $\mathcal{G}$ be a nonempty building set of subvarieties of $Y$. Then the wonderful compactification $Y_{\mathcal{G}}$ is a nonsingular variety. Moreover, for any $G\in \mathcal{G}$ there is a nonsingular divisor $D_{G}\subset Y_{\mathcal{G}}$ such that
\begin{enumerate}
\item the union of these divisors is $Y_{\mathcal{G}}\setminus Y^{0}$;
\item any set of these divisors meet transversally. An intersection of divisors $D_{T_1}\cap \cdots \cap D_{T_r}$ is nonempty  exactly when $\{T_1, \ldots, T_r\}$ form a $\mathcal{G}$-nest.
\end{enumerate}
\end{thm}  

\begin{thm}\label{t2}
Let $Y$ be a nonsingular variety and let $\mathcal{G}$ be a nonempty building set of subvarieties of $Y$.
Let the building set  $\mathcal{G} = \{G_1, \ldots , G_{Q}\}$ is ordered such  that the sets of subvarieties $\{G_1, \ldots , G_i\}$  form a building set for any $1\leq i\leq Q$.  Then the iterated blow-ups give the smooth variety
\[
X_{\mathcal{G}}= Bl_{\tilde{G}_{Q}}\cdots Bl_{\tilde{G}_{2}}Bl_{G_1}Y,
\]
where $\tilde{G}_{i}$ is a nonsingular variety obtained as the iterated  dominant transform of $G_{i}$ in $Bl_{\tilde{G}_{i-1}}\cdots Bl_{\tilde{G}_{2}}Bl_{G_1}Y$, $2\leq i\leq Q$ and the smooth manifold $X_{\mathcal{G}}$ coincides with the wonderful compactification $Y_{\mathcal{G}}$. 
\end{thm}

We explain shortly the notions which are used in  these theorems.

\begin{defn}
A simple  arrangement of subvarieties of a nonsingular variety $Y$ is a
finite set $\mathcal{S} = \{S_{i}\}$ of nonsingular closed subvarieties $S_i$, properly contained in $Y$,
that satisfy the following conditions:
\begin{enumerate}
\item  $S_i$  and $S_j$ intersect cleanly;
\item  $S_{i} \cap  S_{j}$  either is equal to some $S_k$ or it is empty.
\end{enumerate}
\end{defn}

Here two closed non-singular subvarieties $S_1$ and $S_2$ in $Y$ are said to intersect cleanly if their intersection is nonsingular and their tangent bundles satisfy $T(S_1\cap S_2) = T(S_1)|_{(S_1\cap S_2)} \cap T(S_2)|_{(S_1\cap S_2)}$.

\begin{defn}
Let $\mathcal{S}$ be an arrangement of subvarieties of $Y$. A subset $\mathcal{G} \subseteq \mathcal{S}$
is called a building set of $\mathcal{S}$ if, for all $S\in \mathcal{S}$, the minimal elements in $\{G\in \mathcal{G} :
 S\subseteq G\}$ intersect transversally and their intersection is $S$.
\end{defn}

Here the subvarieties $S_1, \ldots ,S_k$ intersect transversally if   $k=1$ or 
\[
\text{codim}\big (\bigcap\limits _{i=1}^{k} T(S_i), T(Y)\big ) =\sum\limits_{i=1}^{k}\text{codim}(S_i, Y).
\] 

A finite set $\mathcal{G}$ of nonsingular subvarieties of $Y$ is called a building set if the set of
all possible intersections of collections of subvarieties from $\mathcal{G}$ form an arrangement $\mathcal{S}$ and if $\mathcal{G}$ is a building set of $\mathcal{S}$. Then  $\mathcal{S}$ is called the arrangement
induced by $\mathcal{G}$. 

We point to the following observation which we will find useful.

\begin{lem}\label{building}
Let a finite set  of nonsingular  subvarieties $\mathcal{G}$ of nonsingular variety $Y$ satisfies the following:
\begin{itemize}
\item $\mathcal{G}$ contains all intersections of its elements;
\item  any two elements of $\mathcal{G}$ intersect cleanly.
\end{itemize}
Then  $\mathcal{G}$ is a building set.
\end{lem}

\begin{proof}
In this case the  set $\mathcal{G}$ is  a simple arrangement whose  building set is   $\mathcal{G}$  since for any $S\in \mathcal{G}$ we have that $S$ is the smallest element of the set $\{ G\in \mathcal{G} | S\subseteq G\}$.
\end{proof} 

We recall   also the meaning of blow-ups in Theorem~\ref{t2}.

\begin{defn} Let $Z$  be a nonsingular subvariety of a nonsingular variety $Y$, then let  $Bl_{Z}Y$ be the blow-up of $Y$ along $Z$ and 
and let $\pi  : Bl_{Z}Y \to  Y$  be the canonical projection.  
For any irreducible subvariety $V$ of $Y$ the dominant transform $\tilde{V}$ is defined to be 
\begin{itemize}
\item the strict transform of $V$  if  $V\not\subseteq Z $, which  is the closure  of $\pi ^{-1}(V\setminus (V\cap Z))$ in $Bl_{Z}Y$,
\item   the scheme-theoretic inverse $\pi ^{-1}(V )$  if $V\subseteq Z$.
\end{itemize}
\end{defn}

The reason for introducing the notion of the dominant transform is to correct the fact that the strict transform of a subvariety contained in the center of blow-up is empty. In our applications  it will always be satisfied that  $V\not\subset Z$, so $\tilde{V}$ will always be a strict transform.

We also shortly comment on the proof of Theorem~\ref{t2}, see~\cite{LL}. The proof essentially relies on the result proved in the same paper:  Let  $Y$ be a nonsingular  and  $F$ is a minimal element in the building set $\mathcal{G} = \{G_1, \ldots, G_N\}$ with the induced arrangement $\mathcal{S}$  and $E$ is an exceptional divisor  in the blow up $Bl_{F}Y$.  Then the collection of subvarieties $\tilde{\mathcal{S}}$ in $Bl_{F}Y$  defined by
\[
\tilde{\mathcal{S}} = \{\tilde{S}\}_{S\in \mathcal{S}} \cup \{\tilde{S}\cap E\}_{\emptyset \neq S\cap F\varsubsetneq S}
\]
is an arrangement of subvarieties in $Bl_{F}Y$ and $\tilde{\mathcal{G}}= \{\tilde{G}\}_{G\in \mathcal{G}}$ is a building set in $\tilde{\mathcal{S}}$. 

Then the idea is to  order (partially) the set $\mathcal{G}$ according to the inclusion relation and to start  with the blow up of $Y$  along the minimal  element $F$,  and,  using just   mentioned result, to iteratively repeat   the procedure with the  dominant transform  
$\tilde{\mathcal{G}}$ of $\mathcal{G}$ in $Bl_{F}Y$.  

\subsection{The space $\mathcal{F}_{n}$ as the wonderful compactification on $\bar{F}_{n}$}\label{building}
Let  $\bar{F}_{n}\subset (\C P^{1})^{N}$  be as given in   Proposition~\ref{smooth} and let 
\begin{equation}\label{31}
\hat{F}_{I} =\bar{F}_{n}\cap \{ (c_{ik}:c_{ik}^{'})=(c_{il}:c_{il}^{'}) = (c_{kl}:c_{kl}^{'})=(1:1)\},
\end{equation}
for $I=\{i, k, l\}\in \{I\subset \{1, \ldots, n\},\; |I|=3\}$ and $n\geq 5$.

We take $Y=\bar{F}_{n}$ and take  the building set $\mathcal{G}_{n}$ to be
\begin{itemize}
\item $\mathcal{G}_{n}=\emptyset$ for $n=4$,
\item $ \mathcal{G}_{n} = \{G=\bigcap\limits_{I}\hat{F}_{I}\subset \bar{F}_{n}\}$, that is all possible nonempty intersection of $\hat{F}_{I}$'s.
\end{itemize}

An element $G\in \mathcal{G}_n$   of the form $G =  \hat{F}_{I_1}\cap \cdots \cap \hat{F}_{I_k}$ we denote by $\hat{F}_{I_1, \ldots, I_k}$.  

\begin{lem}
 $\mathcal{G}_n$ is a  building set.
\end{lem}
\begin{proof}
Any  intersection  of  elements from $\mathcal{G}_n$ belongs to  $\mathcal{G}_n$.  The set $\mathcal{G}_n$ is a simple arrangement since obviously intersection of two elements from $\mathcal{G}_n$ is either empty either it belongs to $\mathcal{G}_n$ and any two elements intersect cleanly. We see this   from  the description of the subvarieties ${F}_{I_1, \ldots, I_k}\subset (\C P^{1})^{N}$ given by~\eqref{relat}  and  from the observation that  for $S_1=\hat{F}_{I_{i_1}, \ldots I_{i_k}}, S_2=\hat{F}_{I_{j_1},\ldots I_{j_l}}\in \mathcal{S}$  it holds  
\[
S_1\cap S_{2}= \hat{F}_{I_{i_1}, \ldots ,I_{i_k}, I_{j_1}, \ldots , I_{j_l}},
\]

Lemma~\ref{building}  then implies that $\mathcal{G}_n$ is a building set.
\end{proof}

Next we prove that the building set $\mathcal{G}_n$ satisfies the condition of Theorem~\ref{t2}.

\begin{lem}
The defined building set $\mathcal{G}_n$  can be ordered as  $\mathcal{G}_n=\{G_1, \ldots , G_{Q}\}$ such that the sets of subvarieties $\{G_1, \ldots , G_i\}$  form a building sets for any $1\leq i\leq Q$.
\end{lem}

\begin{proof}
We assign to an element $G=\hat{F}_{I_1, \ldots, I_k}\in \mathcal{G}_n$ the number $\mathfrak{o}(G)$ which is equal to the number of  coordinates of the points  $\bar{F}\subset (\C P^1)^{N}$ which are  determined by the set $I_1\cup \cdots \cup I_k$. In other words $\mathfrak{o}(G)$ is the number of coordinates of the form $(1:1)$  common for  all points from $G$. For example, if $k=2$ and $I_1=345$, $I_2=346$ then $\mathfrak{o}(G) =6$, then for $I_1=345$, $I_2=367$ we have that $\mathfrak{o}(G) = 10$, while in the case $I_1\cap I_2=\emptyset$ in general  we have $\mathfrak{o}(G) = 6$.  We define an equivalence relation on $\mathcal{G}$ by:   $G_1, G_2$ are in relation if and only if $\mathfrak{o}(G_1)=\mathfrak{o}(G_2)$. Denote by $\tilde{\mathcal{G}}_{1}, \ldots, \tilde{\mathcal{G}}_{P}$ the corresponding equivalence classes.  We assume that we  order these  equivalence classes by an order  which is oppositely compatible with the corresponding numbers $\mathfrak{o}(\tilde{\mathcal{G}}_{i})$ that is $i<j$ if and only if $\mathfrak{0}(\tilde{\mathcal{G}}_{i})>\mathfrak{o}(\tilde{\mathcal{G}}_{j})$. It implies that $\tilde{\mathcal{G}}_{1}$ contains only the point $S=(1:1)^{N}$, while $\mathcal{G}_{P}$ is consisted  consists of all elements $\hat{F}_{I}$.

We  further order the elements  of $\mathcal{G}_n$ as follows: $G_1=(1:1)^{N}$, then we put the elements from $\tilde{\mathcal{G}}_{1}$ in an arbitrary order, after that we put the elements from $\tilde{\mathcal{G}}_{2}$ in an arbitrary order and so on, at the end we put the elements of $\tilde{\mathcal{G}}_{P}$,  that is 
$\hat{F}_{I}$ in an arbitrary oder. We denote this order by  $\mathcal{G}_n=\{G_1, \ldots ,G_{Q}\}$.  Since the elements of $\mathcal{G}$ intersect cleanly it follows that  the set   $\{G_1, \ldots, G_{i}\}$ is a building set for any $1\leq i\leq Q$.
\end{proof}

By $\mathcal{F}_{n}$  we denote the smooth, compact manifold $Y_{\mathcal{G}}$  which is the  wonderful compactification with the building set $\mathcal{G} = \mathcal{G}_{n}$ and $Y= \bar{F}_{n}$.  

\begin{rem}
Note that for $n=5$ the building set $\mathcal{G}_{5}$ consists of one point $P=((1:1), (1:1), (1:1))$ and, therefore, 
$\mathcal{F}_{5}= Bl_{P}\bar{F}_{5}$, compare to~\cite{BT-2}. For $n=6$, the description  of  $\mathcal{F}_{6}$ is no more trivial and allows us to demonstrate the general approach.
\end{rem}

We have the situation that we are given a smooth manifolds  $F_{n}\subset (\C P^{1})^{N}$, which is an open subset in $\bar{F}_{n}\subset (\C P^{1})^{N}$ and the group of automorphisms $\mathcal{A}= \{f_{ij, kl}\}$ for  $F_{n}$ which are induced by the transition maps between the charts $M_{ij}$ and $M_{kl}$ for $G_{n,2}$.   The manifold $\mathcal{F}_{n}$ which is a compactification for $F_{n}$ satisfies the desired property stated in the introduction :

\begin{thm}\label{wcom}
The homeomorphisms of $F_{n}$ given by the set    $\mathcal{A}$  extend to the homeomorphisms of $\mathcal{F}_{n}$.
\end{thm}

\begin{proof}
Since it holds $f_{ij, kl} = f^{-1}_{12, ij}\circ f_{12, kl}$,  it is enough to prove the statement for the homeomorphisms $f_{12, ij}$. We demonstrate the proof for the  homeomorphisms $f_{12, 1j}$ given by Lemma~\ref{jedan}, the other cases go in an analogous way.
We discuss first  the homeomorphic extension  of the homeomorphisms $f_{12, ij}$ to the boundary $\bar{F}_{n}\setminus F_{n}$.
This  boundary  is given by the conditions  $(c_{kl}:c_{kl}^{'}) = (1:0)$ or $(0:1)$ or $(1:1)$ for some $3\leq k<l\leq n$. We analyze each of these cases using the equations   $c_{kl}c_{kp}^{'}c_{lp}=c_{kl}^{'}c_{kp}c_{lp}^{'}$ which define $\bar{F}_{n}$  and the expressions for $f_{12,1j}$ given by Lemma~\ref{jedan}.

\begin{itemize}
\item  We first assume that $k\neq j$.  
 For $c_{kl}^{'}=0$   we have that $c_{kp}^{'} = 0$ or $c_{lp}=0$,  which implies that   $d_{kl}^{'} = 0$ and, $d_{kp}^{'} = 0$ or $d_{lp}=0$.    For $c_{kl} = 0$ we have that $c_{kp}=0$ or $c_{lp}^{'}=0$,  which gives $d_{kl} =0$ and,  $d_{kp}=0$ or $d_{lp}^{'} =0$. For $c_{kl}=c_{kl}^{'}$ we have $(c_{lp}:c_{lp}^{'})=(c_{kp}:c_{kp}^{'})$,  which implies $d_{kl}=d_{kl}^{'}$ and $(d_{lp}:d_{lp}^{'})=(d_{kp}:d_{kp}^{'})$.  
\item We assume now that  $k=j$.  If  $c_{jl} = c_{jp}=0$ we have that $d_{2l}=d_{2p}=0$, while for $c_{jl}=c_{lp}^{'}=0$ we have that
$d_{2l}=0$ and $d_{lp}^{'}=0$. 
If $c_{jl}^{'}= c_{jp}^{'}=0$, then $(c_{lp}:c_{lp}^{'})$ can be an arbitrary element form $\C P^1$, while we have that $(d_{2l}:d_{2l}^{'})= (d_{2p}:d_{2p}^{'}) = (d_{lp}:d_{lp}^{'}) = (1:1)$ . Note that in this case $f_{12, 1j}$ extends to such elements of the boundary but it can not be a homeomorphism. If $c_{jl}^{'} = c_{lp}=0$ we obtain that $d_{2l}=d_{2l}^{'}$ and $d_{lp} = 0$. For $c_{jl}=c_{jl}^{'}$ we have that $(c_{lp}:c_{lp}^{'})=(c_{jp}:c_{jp}^{'})$ which implies $d_{2l}^{'}=0$ and $(d_{lp}:d_{lp}^{'})=(d_{jp}:d_{jp}^{'})$.
\end{itemize}

Altogether we conclude that $f_{12, 1j}$ can not be continuously extended to the subvarieties $\hat{F}_{I} \subset \bar{F}_{n}\setminus F_{n}$, $I=\{j,l,p\}$ given by  $(c_{jl}:c_{jl}^{'}) = (c_{jp}:c_{jp}^{'}) = (c_{lp}:c_{lp}^{'})=(1:1)$ and that it can be continuously,  but not homeomorphically,  extended to the subvarieties $\breve{F}_{I}\subset \bar{F}_{n}\setminus F_{n}$, $I=\{j,l,p\}$ given by  $(c_{jl};c_{jl}^{'})=(c_{jp}:c_{jp}^{'})=(1:0)$. We denote  by $\mathcal{G}(j)$ the family of subvarieties consisting of  all possible non-empty intersections of the subvarieties $\hat{F}_{I}$ and by $\mathcal{H}(j)$ the family of all possible non-empty intersections of the subvarieties $\breve{F}_{I}$.  From the previous discussion we see that the map $f_{12, 1j}$ extends homeomorphically to the complement   in $\bar{F}_{n}$ of the union of the    subvarieties from $\mathcal{G}(j)$ and $\mathcal{H}(j)$, that is to $\bar{F}_{n}\setminus (\mathcal{G}(j)\cup \mathcal{H}(j))$.

Moreover, note that the preimages of the subvarieties $\hat{F}_{I}$   by these  extensions  of  $f_{12,1j}$  are the subvarieties $\breve{F}_{I}$.

We extend  a  homeomorphism $f_{12, 1j}$ to  the homeomorphism $\tilde{f}_{12, 1j} :   \mathcal{F}_{n}\to \mathcal{F}_{n}$  as follows:
\begin{itemize}
\item On the complement of the union of  subvarieties from $\mathcal{G}(j)$ and $\mathcal{H}(j)$ the map $\tilde{f}_{12, 1j}$ is  given by the natural  homeomorphic extension  of $f_{12, 1j}$. 
\item Let $\tilde{S}\in \mathcal{H}(j)$, then  $\tilde{S}=\tilde{F}_{I_1}\cap\cdots \cap \tilde{F}_{I_k}$ for some $I_1, \ldots , I_{k}\subset \{I\subset \{1,\ldots, n\}, \; |I|=3, \; j\in I\}$ and let $\hat{S}\in \mathcal{G}(j)$ is given by $\hat{S}= \hat{F}_{I_1}\cap\ldots\cap \hat{F}_{I_k}$. Then  we define $\tilde{f}_{12, 1j}$ to map homeomorphically $\tilde{S}$ to  an exceptional divisor  $E(\hat{S})$  for $\hat{S}$ in $\mathcal{F}_{n}$. This can be naturally done because of the previous observation on behavior of an extension of the map $f_{12, 1j}$ on the subvarieties $\breve{F}_{(j,l,p)}$.
\item Let $E(\hat{S})\subset \mathcal{F}_{n}$ be an exceptional divisor for $\hat{S}\in \mathcal{G}(j)$ where  $\hat{S}=\hat{F}_{I_1}\cap\ldots\cap \hat{F}_{I_k}$,    $I_1, \ldots , I_{k}\in  \{I\subset \{1,\ldots, n\}, \; |I|=3, \;  j\in I\}$. We define  $\tilde{f}_{12, 1j}$ to map homeomorphically $E(\hat{S})$ to   $\tilde{S}=\tilde{F}_{I_1}\cap\cdots \cap \tilde{F}_{I_k}$, as the inverse of the previously defined extension $\tilde{f}_{12, 1j} : \tilde{S}\to E(\hat{S})$.
\end{itemize}

\end{proof}

 
\section{$\mathcal{F}_{n}$ and moduli space $\overline{\mathcal{M}}_{0,n}$}
\subsection{The main result}
We denote by $\mathcal{M}_{0,n}$, as usual, the moduli space of curves  of genus $0$ with $n$ marked distinct points.  The space $\mathcal{M}_{0,n}$ parametrizes $n$-tuples of distinct points on the Riemann sphere $\C P^1$ up to biholomorphisms, that is
\[
\mathcal{M}_{0,n} = ((\C P^{1})^{n}\setminus \Delta)/PGL_{2}(\C),
\]
where  $\Delta = \bigcup_{i\neq j}\{(x_1, \ldots, x_n)\in (\C P^{1})^{n} | \; x_i=x_j\}$.  It  follows that $\mathcal{M}_{0,n}$ can be identified with
\[
\mathcal{M}_{0,n} = \{ (x_1, \ldots , x_{n-3})\in \C P^{n-3} |\;  x_i\neq 0,1,\infty, \; x_i\neq x_j\}.
\]
For example $\mathcal{M}_{0,3}$ is a points, while $\mathcal{M}_{0,4}=\C P^{1}\setminus \{0,1, \infty\}$.
Note that the moduli space $\mathcal{M}_{0,n}$ coincides with our space of parameters of the main stratum $F_{n}$, compare to~\eqref{spmain}.

The moduli space $\overline{\mathcal{M}}_{0,n}$ is the space of biholomorphism classes of stable curves of genus $0$ with $n$ marked distinct points. It is a compact, complex manifold of dimension $n-3$ in which $\mathcal{M}_{0, n}$ is a Zariski-open subset.
The moduli space  $\overline{\mathcal{M}}_{0,n}$ is a compactification of $\mathcal{M}_{0,n}$ known as the Grothendieck-Deligne-Knudsen-Mumford compactification.

 In~\cite{Keel} Keel  has given an alternative, to that of Grothendieck-Knudson,  construction of  the smooth complete variety $\overline{\mathcal{M}}_{0,n}$.  It has been then noted  in~\cite{LL}  that Theorem~\ref{t2} when applied to his  construction  immediately leads that $\overline{\mathcal{M}}_{0,n}$ is a wonderful compactification  $Y_{\mathcal{G}}$ where $Y =(\C P^1)^{n-3}$ and  the building  set $\mathcal{G}$ consists of the set of all diagonals
 and augmented diagonals. More precisely, $\mathcal{G}$ consists of 
\[
\Delta _{I}==\{(c_4,\ldots c_n)\in (\C P^{1})^{n-3} | c_i=c_j \;  \text{for all} \; i,j\in I\}, 
\]
\[
\Delta _{I,a} =\{(c_4,\ldots c_n)\in (\C P^{1})^{n-3} | c_i=a \; \text{for all} \; i\in I\},\]
where $I\subseteq \{4, \ldots , n\}$, $|I|\geq 2$ and $a\in \{0, 1, \infty\}$. The corresponding arrangement  is the set of all intersections of elements in $\mathcal{G}$. 

Using this,  we prove that our compactification $\mathcal{F}_{n}$ of $F_{n}$, which is  obtained as the wonderful compactification of the building  set $\mathcal{G}_n$     consisting of the  nonsingular subvarieties in  $\bar{F}_{n}$ given in Section~\ref{building}, coincides with the   Grothendieck-Deligne-Knudsen-Mumford compactification of $\mathcal{M}_{0,n} = F_{n}$.

\begin{thm}~\label{modul}
The manifold $\mathcal{F}_{n}$ is diffeomorphic to the manifold $\overline{\mathcal{M}}_{0,n}$, $n\geq 4$.
\end{thm}

\begin{rem}
Recall that for $n=4$ it is  already  known that $\mathcal{F}_{4}= \C P^1  = \overline{\mathcal{M}}_{0,4}$. Also, for $n=5$  it has already been  noted  in~\cite{BT-2}, Remark 7.13  that $\mathcal{F}_{5}$ coincides with $\overline{\mathcal{M}}_{(0,5)}$.
\end{rem}

\subsection{$\mathcal{F}_{6}$ and $\overline{\mathcal{M}}_{0,6}$}

For the sake of clearness we first elaborate Theorem~\ref{wcom}  and prove Theorem~\ref{modul} in the case of Grassmannian $G_{6,2}$. In addition, $G_{6,2}$ is of particular importance in algebraic geometry being one of the six   Severi varieties,~\cite{LM}.
For $n=6$  we start with  the manifold 
\[
\bar{F}_{6}=\{((c_{34}:c_{34}^{'}), (c_{35}:c_{35}^{'}), (c_{36}:c_{36}^{'}), (c_{45}:c_{45}^{'}), (c_{46}:c_{46}^{'}), (c_{56}:c_{56}^{'}))\in (\C P^{1})^{6},\]
\[c_{34}^{'}c_{35}c_{45}^{'} = c_{34}c_{35}^{'}c_{45},\;  c_{34}^{'}c_{36}c_{46}^{'} = c_{34}c_{36}^{'}c_{46},
\]
\[  c_{35}^{'}c_{36}c_{56}^{'} = c_{35}c_{36}^{'}c_{56},\;  c_{45}^{'}c_{46}c_{56}^{'} = c_{45}c_{46}^{'}c_{56}, \}
\]

The building set is given by the following subvarieties in $\bar{F}_{6}$:
\[
\hat{F}_{345}=\{((1:1), (1:1), (c_{36}:c_{36}^{'}), (1:1), (c_{46}: c_{46}^{'}), (c_{56}:c_{56}^{'})),  
\]
\[
c_{36}c_{46}^{'} = c_{36}^{'}c_{46},\;   c_
{36}c_{56}^{'} = c_{36}^{'}c_{56},\;   c_{46}c_{56}^{'} = c_{46}^{'}c_{56}\} 
\]
\[
\hat{F}_{346}=\{((1:1), (c_{35}:c_{35}^{'}), (1:1), (c_{45}:c_{45}^{'}), (1: 1), (c_{56}:c_{56}^{'})),
 \]
\[
c_{35}c_{45}^{'} = c_{35}^{'}c_{45},\;  c_{35}c_{56}^{'} = c_{35}^{'}c_{56}, \; c_{45}c_{56}^{'} = c_{45}^{'}c_{56}\} 
\]
\[
\hat{F}_{356}=((c_{34}:c_{34}^{'}), (1:1), (1:1), (c_{45}:c_{45}^{'}), (c_{46}: c_{46}^{'}), (1:1)),
 \]
\[
c_{34}c_{45}^{'} = c_{34}^{'}c_{45},\;  c_{34}c_{46}^{'} = c_{34}^{'}c_{46},\;  c_{45}c_{46}^{'} = c_{45}^{'}c_{46}\}
\]
\[
\hat{F}_{456} = \{(c_{34}:c_{34}^{'}), (c_{35}:c_{35}^{'}), (c_{36}:c_{36}^{'}), (1:1), (1:1), (1:1))\},
\]
\[
c_{34}c_{35}^{'} = c_{34}^{'}c_{35},\;  c_{34}c_{36}^{'} = c_{34}^{'}c_{36},\;  c_{35}^{'}c_{36}=c_{35}c_{36}^{'}\}.
\]
together with the point $S= (1:1)^{6}$.  At this point  any of these two subvarieties  intersect. 


The smooth, compact manifold $\mathcal{F}_{6}$ is a wonderful compactification with the building set $\mathcal{G}_{6}=\{ S, \hat{F}_{345}, \hat{F}_{346}, \hat{F}_{356},\hat{F}_{456}\}$ that is
\begin{equation}\label{bl6}
\mathcal{F}_{6}= Bl_{\tilde{F}_{456}}Bl_{\tilde{F}_{356}}Bl_{\tilde{F}_{346}}Bl_{\tilde{F}_{345}}Bl_{S}\bar{F}_{6}.
\end{equation}
Note that the dominant transform $\tilde{F}_{ijk}$  in $Bl_{S}\bar{F}_{6}$  of any  of the submanifolds  $\hat{F}_{ijk}$, $3\leq i<j<k\leq 6$  intersect an exceptional divisor
$\C P^2$ at an one point.  In addition,   the  four points obtained in such a way are different. It implies that the wonderful compactification given by~\eqref{bl6} does not depend on an  order of the blow ups along the subvarieties  $\tilde{F}_{ijk}$.



We  show that the manifold $\mathcal{F}_{6}$ coincide  with the moduli space $\overline{\mathcal{M}}_{0,6}$. As we already mentioned  the construction from~\cite{Keel}  describes  $\overline{\mathcal{M}}_{0,6}$  as  a sequence of blow-ups , which is then used   in~\cite{LL}  to note  that  $\overline{\mathcal{M}}_{0,6}$ is a wonderful compactification  $Y_{\mathcal{G}}$ where $Y =(\C P^1)^{3}$ and the building  set $\mathcal{G}$ consists of the set of all diagonals
\[
\Delta _{I} =\{ (p_1,p_2, p_{3})\in (\C P^{1})^{3} | p_i =p_j\; \text{for}\; i,j\in I\}
\]
and augmented diagonals
\[
\Delta _{I, a} =\{ (p_1, p_2 , p_{3})\in (\C P^{1})^{3} | p_i = a \; \text{for all} \; i\in I\}
\]
where $I\subset \{1, 2, 3\}$, $|I|\geq 2$ and $a\in A=  \{ 0,1, \infty\}$.

Note that $\Delta_{I}$ is a complex two-dimensional submanifold in $(\C P^{1})^3$ for $|I|=2$,  so the  blowing up $(\C P^{1})^3$ along the diagonals  $\Delta _{I}$, $|I|=2$  leaves $(\C P^{1})^3$ unchanged. Thus, in the wonderful compactification which describes $\overline{\mathcal{M}}_{0,6}$  it is enough to consider,  as a building  set,  the complete diagonal  $\Delta _{123}$ and augmented diagonals $\Delta _{I,a}$.

\begin{thm}
The manifold  $\mathcal{F}_{6}$ is diffeomorphic to   the space $\overline{\mathcal{M}}_{0,6}$.
\end{thm}
\begin{proof}
Let us consider the smooth  map $f : (\C P^{1})^{3} \to (\C P^{1})^{6}$ given by
\[
f((c_{34}:c_{34}^{'}), (c_{35}: c_{35}^{'}), (c_{36}:c_{36}^{'})) = 
\]
\[
((c_{34}:c_{34}^{'}), (c_{35}: c_{35}^{'}), (c_{36}:c_{36}^{'}), (c_{34}^{'}c_{35} : c_{34}c_{35}^{'}), (c_{34}^{'}c_{36}:c_{34}c_{36}^{'}), (c_{35}^{'}c_{36}:c_{35}c_{36}^{'})).
\]

This map is not defined at the points of the following submanifolds in $(\C P^{1})^{3}$:
\[
\Delta_{\{1,2\}, \infty} = \{(1:0), (1:0), (c_{36}:c_{36}^{'}))\},\;  \Delta_{\{1,2\}, 0} = \{(0:1), (0:1), (c_{36}:c_{36}^{'}))\},
\]
 \[
\Delta_{\{1,3\}, \infty} = \{(1:0), (c_{35}:c_{35}^{'}), (1:0))\},\;  \Delta_{\{1,3\}, 0} = \{(0:1), (c_{35}:c_{35}^{'}), (0: 1))\},
\]
 \[
\Delta_{\{2,3\}, \infty} = \{(c_{34}:c_{34}^{'}), (1:0), (1:0))\}, \;  \Delta_{\{2,3\}, 0} = \{(c_{34}:c_{34}^{'}), (0:1), (0:1))\}.
\]
Then for $P = ((1:0), (1:0), (1:0))$ and  $Q= ((0:1), (0:1), (0:1))$
\[
\mathcal{G}^{'} =  \{\Delta_{\{i,j\}, a}, a=0, \infty\} \cup  \{P, Q\}
\]
is a building set  in $(\C P^{1})^{3}$  and  we  consider the wonderful compactification $Z=(\C P^{1})^{3}_{\mathcal{G}^{'}}$.  The map $f$ extends to the diffeomorphism  $\bar{f}$  between $Z$ and $\bar{F}_6$. For example the points of the divisor
 $\C P^1$  along the submanifold  $\Delta_{\{1, 2\},0}$,  the map $\bar{f}$ maps  by
\[
\bar{f}( (1:0), (1:0), (c_{36}:c_{36}^{'}), (x_1:x_2)) = ((1:0), (1:0), (c_{36}:c_{36}^{'}), (x_1:x_2), (0:1), (0:1),
\]
while  the points $(x_1:x_2:x_3)$  of the divisor $\C P^2$  at  the point $P$,  the map $\bar{f}$ maps by 
\[
\bar{f}(P, (x_1:x_2:x_3)) = ((1:0), (1:0), (1:0), (x_1:x_2), (x_1:x_3), (x_2:x_3)).
\]
Since in the neighborhood $((1 : c_{34}^{'}), (1 : c_{35}^{'}), (1 : c_{36}^{'}))$  of $P= ((1:0), (1:0), (1:0))$ it holds 
\[
c_{34}^{'}x_2 = c_{35}^{'}x_1, \; c_{34}^{'}x_3 = c_{36}^{'}x_1, \; c_{35}^{'}x_3=c_{36}^{'}x_2,
\]
it follows that the points $((1:0), (1:0), (1:0), (x_1:x_2), (x_1:x_3), (x_2:x_3))$ belong to $\bar{F}_6$. 

In order to finish the proof it is left to note   that the wonderful compactification for $\bar{F}_6$ with the  building set consisting of $\bar{F}_{ijk}$ and $S$ corresponds to the wonderful compactification for  $Z$ with the building set consisting of  $\Delta _{\{1,2\},1}$, $\Delta _{\{1,3\}, 1}$, $\Delta_{\{2,3\},1}$,
$R=((1:1), (1:1), (1:1))$ 
and the diagonal $\Delta _{123}$. It follows  that $\mathcal{F}_{6}$ and $\overline{\mathcal{M}}_{0,6}$ are diffeomorphic.  
\end{proof}

\subsection{The proof of the main result}

\begin{proof} (of Theorem~\ref{modul}). The proof proceeds in an analogous way as for $n=6$.  According to~\eqref{relatmain1},~\eqref{relatmain2} and~\eqref{spmain} one can start with the smooth map $f : (\C P^{1})^{n-3} \to (\C P^{1})^{N}$, $N = {n-2\choose 2}$  given by
\[
f((c_{34}:c_{34}^{'}),\ldots , (c_{3n}:c_{3n}^{'})) = 
\]
\[
((c_{34}:c_{34}^{'}), \ldots , (c_{3n}:c_{3n}^{'}), (c_{34}^{'}c_{35} : c_{34}c_{35}^{'}),\ldots , (c_{3n-1}^{'}c_{3n}:c_{3n-1}c_{3n}^{'})).
\]
The map $f$ is not defined at the points of the submanifolds $G_{pq}, G_{pg}^{'}\subset (\C P^{1})^{n-3}$, $4\leq p<q\leq n$  given by
\[
G_{pq} = \{((c_{3i}:c_{3j}))\in (\C P^{1})^{n-3} | (c_{3p}:c_{3p}^{'})=(c_{3q}:c_{3q}^{'})=(1:0)\}
\]
\[
G_{pq}^{'} = \{((c_{3i}:c_{3j}))\in (\C P^{1})^{n-3} | (c_{3p}:c_{3p}^{'})=(c_{3q}:c_{3q}^{'})=(0:1)\}.
\]
From~\eqref{relatmain1} and~\eqref{relatmain2} it obviously follows that the map $f$ gives a diffeomorphism between $(\C P^{1}_{A})
^{n-3}\setminus \Delta$ and the space of parameters of the main stratum $F_{n}$.

It is easy to verify that the set $\mathcal{G}^{'}$ of all possible intersections of the subvarieties $G_{pq}, G_{pq}^{'}$ form a building set.
Let $Z$ be a smooth manifold obtained as the wonderful compactification of $(\C P^{1})^{n-3}$ with the building set $\mathcal{G}^{'}$ that is
$Z= (\C P^{1})^{n-3}_{\mathcal{G}^{'}}$. Then, as in the case $n=6$, we see that  the map $f$ extends to the diffeomorphism $\bar{f}$ between $Z$ and $\bar{F}_{n}$.

Let further 
\[
H_{pq} = \{((c_{3i}:c_{3j}))\in (\C P^{1})^{n-3} | (c_{3p}:c_{3p}^{'})=(c_{3q}:c_{3q}^{'})=(1:1)\}
\]
and let $\tilde{H}_{pq}$ be a proper transform of $H_{pq}$ in $Z$. The set $\mathcal{G}^{"}$ of all possible intersections of subvarieties $\tilde{H}_{pq}$ is  a building. It is explained   in~\cite{LL}, subsection 4.4,  that the manifold $\overline{\mathcal{M}}(0, n)$ coincides  with the wonderful compactification $Z_{\mathcal{G}^{"}}$. It is left to note that diffeomorphism $\bar{f}$ extends to the diffeomorphism between  the wonderful compactification $Z_{\mathcal{G}^{''}}$ and    the wonderful compactification $(\bar{F}_{n})_{\mathcal{G}}$, where the building set $\mathcal{G}$ is given by all possible intersections of the subvarieties~\eqref{31}. Thus, the smooth manifolds $\overline{\mathcal{M}}(0,n)$ and $\mathcal{F}_{n}$ coincides.
\end{proof}

\section{ $\mathcal{F}_{n}$ and  Chow quotient $G_{n,2}\!/\!/(\C ^{\ast})^{n}$}

\subsection{Basic facts on  Chow varieties and Chow quotient}
We follow the monograph~\cite{GKZ}, Chapter 4  to recall the basic facts on Chow varieties, while for the  notion of Chow quotient  $G_{n,k}\!/\!/(\C ^{\ast})^{n}$ we follow the paper of Kapranov~\cite{Kap}. The idea behind the definition of the Chow quotient  is the construction from algebraic geometry  known as  the Chow varieties, that is compact varieties whose points parametrize algebraic cycles in a given variety  of the same dimension and degree. The Chow variety for $G_{n,k}$  which is needed  in the definition of the Chow quotient  can be defined as follows. Let $\delta \in H_{2(n-1)}(G_{n,k}, \Z)$ be the homology class of the closure of a  generic $(\C ^{\ast})^{n}$ - orbit in $G_{n,k}$ and   let  $C_{2(n-1)}(G_{n,k}, \delta)$ denotes the set of all algebraic cycles in $G_{n,k}$ of dimension $2(n-1)$ whose homology class is $\delta$.  The Grassmann  manifolds $G_{n,k}$ embeds into $\C P^{N}$, $N={n\choose k}-1$ via Pl\"ucker embedding, so let  $d\in H_{2(n-1)}(\C P^{N}, \Z)\cong \Z$ be the image of the class $\delta$ under this embedding. Now, consider the set $G(N, d, 2(n-1))$  of algebraic cycles in $\C P^{N}$ of dimension $2(n-1)$ and degree $d$.  In other words, one considers algebraic cycles  whose multiplicity in $H_{2(n-1)}(\C P^{N}, \Z)$, regarded to the canonical generator, is $d$.  Denote by $\mathcal{B}$ the coordinate ring of $G_{n,k}$ via the Pl\"ucker embedding, that is the   quotient of the polynomial ring $\C [z_{1}, \ldots z_{N+1}]$  by the Pl\"ucker relations  and denote by $\mathcal{B}_{d}$  a complex linear subspace which is given by the  homogeneous part of $\mathcal{B}$ of degree $d$.  By the theorem of Chow and van der Waerden, the set $G(N, d, 2(n-1))$ becomes a closed projective algebraic variety, in particular compact, via Chow embedding $G(N, d, 2(n-1))\to P(\mathcal{B}_{d})$.  The set $C_{2(n-1)}(G_{n,k}, \delta)$ endowed with the  resulting structure  of the algebraic variety via  $C_{2(n-1)}(G_{n,k}, \delta)\subset G(N, d, 2(n-1))$ is the needed Chow variety for $G_{n,k}$.

  In more detail, the mentioned Chow embedding is defined as follows, see~\cite{GKZ}.  For any irreducible algebraic cycle  $X\in G(N, d, 2(n-1))$  one can consider the set $\mathcal{Z}(X)$  of all $(N-2(n-1)-1)$ - dimensional projective subspaces $L$ in $\C P^{N}$ which intersect $X$.  The set $\mathcal{Z}(X)$ is a  subvariety in  the  Grassmannian $G(N, N-2(n-1)+1)$. It can be  proved  that $\mathcal{Z}(X)$ is defined by some element $R_{X}\in \mathcal{B}_{d}$ which is unique up to constant factor and $R_{X}$ is called the Chow form of $X$.  If $X$ is not irreducible cycle then $X=\sum a_iX_i$, where $X_i$ are $2(n-1)$-dimensional closed irreducible varieties and $a_i$ are non-negative integer coefficients, and the Chow form for $X$ is defined by $R_X=\prod R_{X_i}^{a_i}\in \mathcal{B}_{d}$. The map $X\to R_{X}$ defines an embedding of $G(N, d, 2(n-1))$ into the projective space $P(\mathcal{B}_{d})$ which is called the Chow embedding.

In order to define the Chow quotient one considers  the natural map 
\[
W/(\C ^{\ast})^{n} \to C_{2(n-1)}(G_{n,k}, \delta), \;\;  x \to \overline{(\C ^{\ast})^{n}\cdot x},
\]  
where $W$ is the main stratum in $G_{n,k}$ which consists of the points whose all Pl\"ucker coordinates are non-zero.
By definition, the Chow quotient  $G_{n,k}\!/\!/(\C ^{\ast})^{n}$ is the closure of the  image of this map.

We recall the following results from~\cite{Kap}, Propositions (1.2.11), (1.2.15)  which give the  description of  the build up components  of  $W/(\C ^{\ast})^{n}$ in  $G_{n,k}\!/\!/(\C ^{\ast})^{n}$. 

\begin{prop}\label{norost}
The algebraic cycles in the Chow  quotient  $G_{n,k}\!/\!/ (\C ^{\ast})^{n}$ are of the form
$Z= \sum _{i}Z_{i}$, where $Z_i$ are the closures of $(\C ^{\ast})^{n}$-orbits in $G_{n,k}$ such that the matroid polytopes $\mu (Z_i)$  give a polyhedral decomposition of $\Delta_{n,k}$.   
\end{prop} 

We point out that  the matroid polytopes defined in~\cite{Kap} in the case $G_{n,2}$ coincide with our admissible polytopes~\cite{BT-n2}. 
Note that,  in our terminology,  the Chow quotient  gives  actually  a compactification of the space of parameters $F_n$ of the main stratum.   

We want to emphasize that the Chow quotient $X\!/\!/H$  can be defined for any complex projective variety $X\subset \C P^{N}$  with an action of an algebraic group $H$, see~\cite{Kap}. Namely,  the orbit closure $\overline{H\cdot x}$  is a compact subvariety in $X$ for any point $x\in X$ and 
for  a small Zariski open $H$-invariant  subset $U\subset X$ which consists of generic points all  varieties $\overline{H \cdot x}$ for $x\in U$  have the same dimension $m$ and represent the same homology class $\delta \in H_{2m}(X, \Z)$. One can consider  the Chow variety $C_{2m}(X, \delta)\subset  G(N, d, 2m)$, where $d$ is an image of $\delta$ by an embedding  $X\to \C P^{N}$,  and the Chow quotient $X\!/\!/H$   is the closure of the image of the map $U/H \to C_{2m}(X, \delta)$ defined by  $x\to \overline{x\cdot H}$.


Using Gel'fand-MacPherson construction  Kapranov  in~\cite{Kap}   proved that for $k=2$ the Chow quotient $G_{n,2}\!/\!/(\C ^{\ast})^{n}$ is isomorphic to the Chow quotient  $(\C P^{k-1})^{n}\!/\!/GL(k)$. Appealing on this,  he constructed an  isomorphism between the Chow quotient    $G_{n,2}\!/\!/(\C ^{\ast})^{n}$ and the Grothendieck-Knudsen compactification $\overline{\mathcal{M}}_{0,n}$. In addition using the construction of such isomorphism Kapranov in~\cite{Kap}  provided the description of the quotient $G_{n,2}\!/\!/(\C ^{\ast})^{n}$ as the sequence of blow ups along some specified subvarieties in $\C P^{n-3}$. 

 In conclusion, we want to point  that  Theorem~\ref{modul} implies that our  construction of  the   space $\mathcal{F}_{n}$  provides the new,  purely topological   approach for the description  of the Chow quotient $G_{n,2}\!/\!/(\C ^{\ast})^{n}$. 

\subsection{The structures of   $G_{n,2}\!/\!/(\C ^{\ast})^{n}$ and  $G_{n,2}/T^n$}

In~\cite{BT-n2} we  described the orbit space $G_{n,2}/T^n$ in terms of CW complex of admissible polytopes and the universal space of parameters $\mathcal{F}_{n}$. In this description, fundamental  role has the chamber decomposition of $\Delta_{n,2}$ induced by the admissible polytopes.   Using the results of~\cite{Kap}, first we describe  the Chow quotient    $G_{n,2}\!/\!/(\C ^{\ast})^{n}$ in terms of cort\'eges of admissible polytopes which  give the polyhedral decomposition for $\Delta _{n,2}$ and the spaces of parameters of these polytopes. In addition,  we describe    $G_{n,2}\!/\!/(\C ^{\ast})^{n}$ in terms of the virtual spaces of parameters for the admissible polytopes which form a  $(n-1)$-dimensional chamber in $\Delta_{n,2}$.

Let $\mathcal{P}$ denotes  the  family of admissible polytopes of dimension $n-1$  for the standard $T^n$ -  action on $G_{n,2}$. Let  the set $\{ \mathcal{P}_{1}, \ldots , \mathcal{P}_{l}\}$ consists of all subfamilies $\mathcal{P}_{i}=\{P_{i_1}, \ldots , P_{i_s}\} \subset \mathcal{P}$ such that the polytopes  $P_{i_1}, \ldots , P_{i_s}$  give a polyhedral decomposition for $\Delta _{n,2}$, that is 
$\cup _{j=1}^{s}P_{i_j}=\Delta _{n,2}$ and $\stackrel{\circ}{P}_{i_j}\cap \stackrel{\circ}{P}_{i_k} = \emptyset$ for any $1\leq j<k\leq s$.  

We assign to  a family $\mathcal{P}_{i} = \{P_{i_1}, \ldots , P_{i_s}\}$ the set $\mathcal{W}_{i}= \{W_{i_1},\ldots, W_{i_s}\}$, where  $W_{i_j}$ are the strata in $G_{n,2}$ which  correspond to the admissible polytopes $P_{i_j}$, that is $\mu (W_{i_j}) = \stackrel{\circ}{P}_{i_j}$.  By factorizing  these strata by $(\C ^{\ast})^{n}$-action,  we can  assign to any $\mathcal{P}_{i}$ the set $\{ F_{i_1}, \ldots , F_{i_s}\}$. 

Further, we introduce the space of parameters $\mathcal{F}_{i}$  of a family $\mathcal{P}_{i}$ as the multiset product  of the set 
$\{F_{i_1}, \ldots , F_{i_s}\}$, meaning that it is  a topological space homeomorphic to the direct product $\mathcal{F}_{i} = F_{i_1} \times \cdots \times F_{i_s}$.    

It follows  from Proposition~\ref{norost} that the Chow quotient $G_{n,2}\!/\!/(\C ^{\ast})^{n}$  is the {\it disjoint} union  of the connected components $\mathcal{C}_{i}$, each  consisting  of algebraic cycles determined by a family $\mathcal{P}_{i}$, $1\leq i\leq l$. In particular,   the complement of $F_{n}$ in $G_{n,2}\!/\!/(\C ^{\ast})^{n}$ is the disjoint union of the components  $\mathcal{C}_{i}$ for which  $\mathcal{P}_{i} \neq \{\Delta_{n,2}\}$.

\begin{prop}\label{bij}
There is a bijection between $\mathcal{F}_{i}$ and $\mathcal{C}_{i}$ for any $1\leq i\leq l$.
\end{prop}

\begin{proof}
We define the bijection $g_i : \mathcal{F}_{i}\to \mathcal{C}_{i}$ as follows:
\[
g_{i}(c_{i_1}, \ldots, c_{i_s})= Z_{i_1}(c_1)+\ldots + Z_{i_s}(c_{i_s}), 
\]
where $Z_{i_j}(c_{i_j})$ is the $(\C ^{\ast})^{n}$-orbit from the stratum $W_{i_j}$ which is determined by the parameter $c_{i_j}$.
 \end{proof}

The   Chow quotient $G_{n,2}\!/\!/(\C ^{\ast})^{n}$  and the  complement of $F_{n}$ in  $G_{n,2}\!/\!/(\C ^{\ast})^{n}$ can be interpreted in the following way as well.

Let $P_{\sigma}$ is an admissible polytope and consider the set $\mathcal{P}_{\sigma} = \{\mathcal{P}_{\sigma, 1}, \ldots , \mathcal{P}_{\sigma, s}\}\subset \mathcal{P}$ which consists of  all  decompositions $\mathcal{P}_{\sigma, i}$ of $\Delta _{n,2}$ which contain $P_{\sigma}$,  that is $\mathcal{P}_{\sigma, i}\in \mathcal{P}$ if and only if  $P_{\sigma} \in \mathcal{P}_{\sigma, i}$.

Let $\tilde{Z} _{\sigma, i}\subset  G_{n,2}\!/\!/(\C ^{\ast})^{n}$ be the family of algebraic cycles determined by the decompositions 
$\mathcal{P}_{\sigma, i} = \{P_{\sigma_{i_1}}, \ldots, P_{\sigma _{i_q}}\}$. These cycles are,  by Proposition~\ref{bij},  of the form
\[
Z_{\sigma _{i_1}}(c_{\sigma _{i_1}})+\ldots + Z_{\sigma, _{i_q}}(c_{\sigma _{i_q}}),
\]
for $(c_{\sigma _{i_1}}, \ldots , c_{\sigma _{i_q}}) \in F_{\sigma _{i_1}}\times \cdots \times F_{\sigma _{i_q}}$, where 
$Z_{\sigma _{i_j}}(c_{\sigma _{i_j}})$ is the closure of the algebraic torus orbit in $W_{\sigma _{i_j}}$ and this orbit is determined by $c_{\sigma _{i_j}}\in F_{\sigma _{i_j}} = W_{\sigma _{i,j}}/(\C ^{\ast})^{n}$. 
Let further
\[
\tilde{Z}_{\sigma} = \bigcup\limits_{i=1}^{s}\tilde{Z}_{\sigma, i}.
\]
\begin{prop}\label{projection}
There exists the projection $p_{\sigma} : \tilde{Z} _{\sigma} \to F_{\sigma}$ for any admissible set $\sigma$.
\end{prop}

\begin{proof} For an algebraic cycle $Z\in \tilde{Z}_{\sigma}$ there exists a $(\C ^{\ast})^{n}$- orbit $Z_{\sigma}$ from the stratum $W_{\sigma}$ whose admissible polytopes is $P_{\sigma}$ and which is an  irreducible summand of $Z$. Since $F_{\sigma} = W_{\sigma}/(\C ^{\ast})^{n}$ there is the canonical $(\C ^{\ast})^{n}$-invariant  projection $q_{\sigma}: W_{\sigma}\to F_{\sigma}$.  We define  $p_{\sigma}(Z) = q_{\sigma}(Z_{\sigma})$.
\end{proof}

Recall~\cite{GM}, ~\cite{BT-n2} that the admissible polytopes define the chamber decomposition for $\Delta_{n,2}$: for some subset $\omega$  of all admissible sets, $C_{\omega}$ is a chamber if 
\[
C_{\omega} = \bigcap\limits_{\sigma \in \omega} \stackrel{\circ}{P}_{\sigma}, \;\; C_{\omega} \cap \stackrel{\circ}{P}_{\sigma}=\emptyset.
\]



\begin{thm}\label{union-chamber}
Let $C_{\omega}\subset \Delta _{n,2}$ be a chamber such that $\dim C_{\omega} = n-1$. Then  $C_{\omega}$ defines the decomposition of  $G_{n,2}\!/\!/(\C ^{\ast})^{n}$ into disjoint union, that is:
\begin{equation}\label{union}
\bigcup\limits_{\sigma \in \omega} \tilde{Z}_{\sigma} = G_{n,2}\!/\!/(\C ^{\ast})^{n}.
\end{equation}
\end{thm}
\begin{proof}
Let $Z\in G_{n,2}\!/\!/(\C ^{\ast})^{n}$. Then $Z$ is determined by some decomposition $\mathcal{P}_{i} = \{ P_{\sigma_{i_1}}, \ldots , P_{\sigma _{i_l}}\}$. Note that for any admissible polytope $P_{\sigma}$ it holds $C_{\omega}\subset \stackrel{\circ}{P}_{\sigma}$ either $C_{\omega}\cap \stackrel{\circ}{P}_{\sigma} = \emptyset$. Thus, there exists $P_{\sigma _{i_{j}}}\in \mathcal{P}_{i}$ such that $C_{\omega}\subset \stackrel{\circ}{P}_{\sigma _{i_{j}}}$ meaning that $\sigma_{i_{j}}\in \omega$.  Therefore,   $Z\in \tilde{Z}_{\sigma _{i_{j}}}$, which implies that $Z\in \cup _{\sigma \in \omega} \tilde{Z}_{\sigma}$.

In order to prove that the union~\eqref{union} is disjoint, we note that  for   $\sigma _1, \sigma _2\in \omega$ we have that  $C_{\omega}\subset \stackrel{\circ}{P}_{\sigma _1}, \stackrel{\circ}{P}_{\sigma _2}$, that is $\stackrel{\circ}{P}_{\sigma _1}\cap \stackrel{\circ}{P}_{\sigma _2}\neq \emptyset.$ It implies that there is no decomposition for $\Delta _{n,2}$ which contains both $P_{\sigma _1}$ and $P_{\sigma _2}$. Thus, there is no algebraic cycle 
$Z\in G_{n,2}\!/\!/(\C ^{\ast})^{n}$ such that $Z\in \tilde{Z} _{\sigma _1}$ and $Z\in \tilde{Z}_{\sigma _2}$.
\end{proof}

 Note that $\tilde{Z}_{\sigma} = F_n$ for $P_{\sigma}= \Delta_{n,2}$, which implies that~\eqref{union} describes as well  the build up components  to $F_{n}$ in $ G_{n,2}\!/\!/(\C ^{\ast})^{n}$,  which are in general no more disjoint.

\begin{rem}
In our papers~\cite{BT-2},~\cite{BT-2nk}, for the purpose of  description of an orbit space $G_{n,2}/T^n$,  we introduced the notion of virtual space of parameters $\tilde{F}_{\sigma}$ for a stratum $W_{\sigma}$. Note that the properties of the spaces $\tilde{Z}_{\sigma}$ formulated  by Proposition~\ref{projection} and Theorem~\ref{union-chamber}  confirm that the spaces $\tilde{Z}_{\sigma}$ correspond to the spaces $\tilde{F}_{\sigma}$  for which $\dim P_{\sigma}=n-1$.   In particular,  Theorem 7 from~\cite{BT-n2}
is an analogue of Theorem~\ref{union-chamber}.
\end{rem}

Theorem~\ref{modul} and the result of~\cite{Kap} that the  manifolds $G_{n,2}\!/\!/(\C ^{\ast})^{n}$  and $\overline{\mathcal{M}}(0,n)$ are isomorphic imply  that the  universal space of parameters $\mathcal{F}_{n}$  describe the topology of the gluing of  the build up components in $G_{n,2}\!/\!/ (\C ^{\ast})^{n}$.  Having the description of $\mathcal{F}_{n}$ as the wonderful compactification of $\bar{F}_{n}$  we give an explicit demonstration of this correspondence  in the cases 
$n=4$ and $n=5$.

\subsection{ $G_{4,2}\!/\!/(\C ^{\ast})^{4}$ and $G_{4,2}/T^4$ }

For  $n=4$ the build up components in $G_{4,2}\!/\!/ (\C ^{\ast})^{4}$  consist of three points  and they glue  together with  $F_{4}\cong \C P^{1}_{A}$  in  $G_{4,2}\!/\!/ (\C ^{\ast})^{4}$ to give $\C P^{1}$.  This  is observed in~\cite{Kap}, but also independently follows from~\cite{BT-1} and Theorem~\ref{modul}. Using our notation, we describe this in the following way.  There are exactly three decompositions of an  octahedron $\Delta _{4,2}$ that is $\mathcal{P}= \{\mathcal{P}_{1}, \mathcal{P}_{2}, \mathcal{P}_{3}\}$  and they are given by the three pairs of four-sided complementary pyramids.  The space of parameters of a stratum  for any of these pyramid is a point. Then Proposition~\ref{bij}  implies  that $(G_{4,2}\!/\!/ (\C ^{\ast})^{4})\setminus F_{4}$ consists of three points, that is $G_{4,2}\!/\!/ (\C ^{\ast})^{4}\cong \C P^1$.  More precisely,  these three points  correspond to the algebraic cycles formed of the $(\C ^{\ast})^{4}$-orbits whose admissible polytopes are complementary pyramids in the octahedron $\Delta _{4,2}$.  In addition, for any pyramid $P_{\sigma}$ we have that $\tilde{Z}_{\sigma}$ is given by the  one algebraic cycle, which implies that $F_{\sigma} = \tilde{Z}_{\sigma}$.

\subsection{   $G_{5,2}\!/\!/(\C ^{\ast})^{5}$ and $G_{5,2}/T^5$}
In the case $n=5$,  we use the results  from~\cite{BT-2}  to formulate the correspondence between $\mathcal{F}_{5}$ and   the Chow quotient $G_{5,2}\!/\!/ (\C ^{\ast})^{5}$.  First,  we immediately have the following:

\begin{lem}
There are $25$ decompositions of the hypersimplex $\Delta_{5,2}$ given by the admissible polytopes for $T^5$-action on $G_{5,2}$.
There are given by the pairs $\{K_{ij}, P_{ij}\}$, $1\leq i<j\leq 5$ and the triples $\{P_{ij}, K_{ij, kl}, P_{kl}\}$., $1\leq i<j\leq 5$, $1\leq k<l\leq5$, $\{i,j\}\cap \{k,l\}=\emptyset$. Here $K_{ij}$ is a polytope with $9$ vertices which does not contain the vertex $\Lambda _{ij}$, then $P_{ij}$ is seven-sided pyramid with the apex $\Lambda _{ij}$, while $K_{ij, kl}$ is a polytope with $8$ vertices which does not contain the vertices $\Lambda _{ij}$ and $\Lambda _{kl}$.
\end{lem}

The space of parameters for  a  admissible polytope $K_{ij}$  is $\C P^{1}_{A}$, while  for the polytopes $P_{ij}$ and $K_{ij, kl}$ it is a point. Then Proposition~\ref{bij} gives:

\begin{cor}
The disjoint build up components for   $F_5$ in $G_{5,2}\!/\!/ (\C ^{\ast})^{5}$ are $\mathcal{C}_{ij} \cong \C P^{1}_{A}$ and the points  $\mathcal{C}_{ij, kl}$ for $1\leq i<j\leq 5$, $1\leq k<l\leq 5$.  A component $\mathcal{C}_{ij}$   consists of the cycles of the form $Z_{ij,9}(c)+Z_{ij,7}$  while a component $\mathcal{C}_{ij,kl}$ consists of the cycle $Z_{ij,7}+Z_{ij, kl} + Z_{kl,7}$, where $c\in \C P^{1}_{A}$. The  irreducible algebraic cycles here  are as follows:
\begin{itemize}
\item  $Z_{ij,9}(c)$  is the closure of a orbit from the stratum whose admissible polytope is $K_{ij}$;
\item $Z_{ij, kl}$    is the closure of the orbit whose admissible polytope  $K_{ij, kl}$;
\item $Z_{ij,7}$ it the closure  of the orbit   whose admissible polytope  $P_{ij}$.
\end{itemize}
\end{cor}

We proved in~\cite{BT-2} that the universal space of parameters $\mathcal{F}_{5}$  is the blow up   of the surface $\bar{F}_{5} = \{(c_1:c_1^{'}), (c_2:c_2^{'}), (c_3:c_3^{'}))\in (\C P^{1})^{3}, \; c_1c_{2}^{'}c_3=c_{1}^{'}c_2c_{3}^{'}\}$ at the point $((1:1), (1:1), (1:1))$.  The identification of $\mathcal{F}_{5}$ with $G_{5,2}\!/\!/ (\C ^{\ast})^{5}$ translates to the gluing of the build up components in $G_{5,2}\!/\!/ (\C ^{\ast})^{5}$  as follows:

\begin{cor}
The gluing of the build up components  in  $G_{5,2}\!/\!/ (\C ^{\ast})^{5}$  corresponds  to the compactification of $F_5\subset \mathcal{F}_{5}$ by the following pattern:
\begin{itemize}
\item the cycles $Z_{ij}(c) =  Z_{ij,9}(c) + Z_{ij,7}\in \mathcal{C}_{ij}$ correspond to the subvarieties  in $ \mathcal{F}_{5}$  as follows:
\begin{itemize}
\item  $Z_{23}(c)  =   ((0:1), (0:1),  (c:c^{'})) $,
\item  $Z_{24}(c) =   ((1:0), (c:c^{'}) , (0:1))$,
\item  $Z_{25}(c)  = ((c:c^{'}) , (1:0), (1:0))$,
\item $Z_{13}(c) = ((1:0), (1:0), (c:c^{'}))$,
\item $Z_{14}(c) =   ((0:1),  (c:c^{'}), (1:0))$,
\item $Z_{15}(c) =  ((c:c^{'}),  (0:1), (0:1))$,
\item $Z_{34}(c) =  ((1:1),(c:c^{'}), (c:c^{'}))$,
\item $Z_{35}(c) =  ((c:c^{'}), (1:1), (c^{'}:c))$,
\item $Z_{45}(c) =    ((c:c^{'}), (c:c^{'}), (1:1))$,
\item $Z_{12}(c)$ to the points from $\C P^{1}_{A}$ of  the divisor $\C P^{1}$.
\end{itemize}
\item the cycles $Z_{ij, kl} = Z_{ij,7} + Z_{ij,kl}+Z_{kl,7} = \mathcal{C}_{ij,kl}$ correspond to the points:

\begin{itemize}

\item  $Z_{14,23} =  ((0:1), (0:1), (1:0)), \; Z_{13,24}  =  ((1:0), (1:0), (0:1))$, 
\item $Z_{15,24} = ((1:0), (0:1), (0:1)),\; Z_{23,45} =  ((0:1), (0:1), (1:1))$,
\item $Z_{24,35} = ((1:0), (1:1), (0:1)),\;Z_{25,34} = ((1:1), (1:0), (1:0))$, 
\item $Z_{15, 23} = ((0:1), (0:1), (0:1)),\; Z_{13, 25} = ((1:0), (1:0), (1:0))$,
\item $Z_{14,25} = ((0:1), (1:0), (1:0)),\; Z_{13,45} = ((1:0), (1:0), (1:1))$;
\item  $Z_{14,35} = ((0:1), (1:1), (1:0)),\;  Z_{15,34} =  ((1:1), (0:1), (0:1))$.
\item $Z_{12, 34}$, $Z_{12, 35}$, $Z_{12, 45}$ to the points $(1:0), (0:1), (1:1)$ respectively  of  the divisor $\C P^1$.
 \end{itemize}
\end{itemize}
\end{cor}

We also directly  establish the correspondence  of the subspaces $\tilde{Z}_{ij,9}, \tilde{Z}_{ij, 7}$ and $\tilde{Z}_{ij,kl}$ with the subspaces in $\mathcal{F}_{5}$.

\begin{cor}
The spaces $\tilde{Z}_{ij,9}$, $\tilde{Z}_{ij, 7}, \tilde{Z}_{ij,kl} \subset G_{5,2}\!/\!/ (\C ^{\ast})^{5}$ are homeomorphic to  $\C P^{1}_{A}$, $\C P^{1}$ and a point respectively. The  corresponding spaces  in $\mathcal{F}_{5}$ are 
\[
\tilde{Z}_{ij, 9} = \mathcal{C}_{ij}, \;\; \tilde{Z}_{ij,kl}= \mathcal{C}_{ij, kl},
\]
 \[
\tilde{Z}_{ij, 7} = \mathcal{C}_{ij} \cup  (\cup _{k, l\in \{1, \ldots ,5\}\setminus \{i,j\}}\mathcal{C}_{ij, kl}).
\]
\end{cor}

Note that the spaces  $\tilde{Z}_{ij,9}, \tilde{Z}_{ij, 7}$ and $\tilde{Z}_{ij, k}$ coincide with the virtual spaces of parameters  $\tilde{F}_{ij, 9}$, $\tilde{F}_{ij, 7}$ and $\tilde{F}_{ij, kl}$ for the corresponding  strata in $G_{5,2}$  from ~\cite{BT-n2}.

 Victor M.~Buchstaber\\
Steklov Mathematical Institute, Russian Academy of Sciences\\ 
Gubkina Street 8, 119991 Moscow, Russia\\
E-mail: buchstab@mi.ras.ru
\\ \\ 

Svjetlana Terzi\'c \\
Faculty of Science and Mathematics, University of Montenegro\\
Dzordza Vasingtona bb, 81000 Podgorica, Montenegro\\
E-mail: sterzic@ucg.ac.me 
\end{document}